\documentclass[11pt]{article}
\usepackage[utf8x]{inputenc}
\usepackage[title]{appendix}
\usepackage[noTeX]{mmap}
\usepackage[left=1in, right=1in, bottom = 1in, top = 1in]{geometry}
\usepackage{hyperref} 
\hypersetup{colorlinks}
\usepackage{amsmath,amsfonts, amsthm, amssymb}
\usepackage{color}
\usepackage{enumerate}
\usepackage{hypbmsec}
\usepackage{bbm}
\usepackage{dsfont}
\usepackage{mathtools}
\usepackage{graphicx}
\usepackage{caption}
\usepackage{subcaption}
\usepackage[section]{placeins}
\usepackage{tikz}
\usetikzlibrary{calc}
\usetikzlibrary{decorations.pathreplacing}
\usetikzlibrary{arrows}
\usepackage{pgffor}
\usepackage{longtable}
\usepackage{comment}
\usepackage{dsfont}
\usepackage{multirow}
\usepackage{rotating}
\usepackage{bm}
\usepackage{xypic}
\usepackage{pgf,tikz,pgfplots}
\pgfplotsset{compat=1.15}
\pgfplotsset{ticks = none}
\newcommand{\SG}{\mathsf{SG}}

\usepackage{fancyhdr}

\usepackage{tocloft}

\theoremstyle{plain}
\newtheorem{theorem}{Theorem}[section]
\newtheorem{proposition}[theorem]{Proposition}

\newtheorem{lemma}[theorem]{Lemma}
\newtheorem{definition}[theorem]{Definition}

\theoremstyle{definition}
\newtheorem{remark}[theorem]{Remark}

\newcommand{\old}[1]{}

\newcommand{\distto}{%
	\mathrel{\vbox{\offinterlineskip\ialign{%
				\hfil##\hfil\cr
				$\scriptscriptstyle\mathrm{d}$\cr
				\noalign{\kern-.05ex}
				$\to$\cr
}}}}
\newcommand{\findimto}{%
	\mathrel{\vbox{\offinterlineskip\ialign{%
				\hfil##\hfil\cr
				$\scriptscriptstyle\mathrm{f.d.}$\cr
				\noalign{\kern-.05ex}
				$\to$\cr
}}}}

\newcommand{\Probto}{%
	\mathrel{\vbox{\offinterlineskip\ialign{%
				\hfil##\hfil\cr
				$\scriptscriptstyle\Prob$\cr
				\noalign{\kern-.05ex}
				$\to$\cr
}}}}
\newcommand{\TVto}{%
	\mathrel{\vbox{\offinterlineskip\ialign{%
				\hfil##\hfil\cr
				$\scriptscriptstyle\mathrm{TV}$\cr
				\noalign{\kern-.05ex}
				$\to$\cr
}}}}

\newcommand{\N}{\mathbb{N}}
\newcommand{\Z}{\mathbb{Z}}

\newcommand{\Prob}{\mathds{P}}

\newcommand{\E}{\mathds{E}}

\newcommand{\eqdist}{%
	\mathrel{\vbox{\offinterlineskip\ialign{%
				\hfil##\hfil\cr
				$\scriptscriptstyle\mathrm{law}$\cr
				\noalign{\kern.2ex}
				$=$\cr
}}}}

\title{Random rotor walks  and i.i.d.~sandpiles on Sierpi\'nski graphs}
\date{\today}
\author{Robin Kaiser and Ecaterina Sava-Huss}

\begin{document}
\maketitle

\begin{abstract}
We prove that, on the infinite Sierpi\'nski gasket graph $\SG$, rotor walk with random initial configuration of rotors is recurrent. We also give a necessary condition for an i.i.d.~sandpile to stabilize. In particular, we prove that an i.i.d.~sandpile with expected number of chips per site greater or equal to three does not stabilize almost surely.
Furthermore, the proof also applies to divisible sandpiles and shows that divisible sandpile at critical density one does not stabilize almost surely on $\SG$.
\end{abstract}



\section{Introduction}\label{sec:intro}

{\it Rotor-router walk} (or shortly {\it rotor walk}) is a deterministic counterpart of a random walk on a graph $G$, defined as follows. For each vertex $x\in G$, we fix a cyclic ordering of its neighbors and we refer to it as the {\it rotor mechanism} at $x$. In addition, each vertex $x$ is equipped with an arrow (or rotor), which points initially to one of the  neighbours of $x$.  A particle (a walker) is located at the beginning at some vertex in $G$ and its location evolves in time  as following. At each time step, the rotor at the particle's location $x$ is incremented to point to the next neighbour in the cyclic ordering of $x$, and then the particle follows this new direction and moves to the vertex the rotor points to.
The rotor walk $(R_t)_{t\in\N}$ is obtained by repeated applications of this rule. 
If the initial direction of the arrows is random, then we call this process {\it random rotor walk} or {\it rotor walk with random initial configuration}. Such processes are interesting because there is remarkable agreement between their behaviour and the expected behaviour of random walks. On the other hand, between these two processes there are also striking differences, for instance in what concerns the recurrence and transience. A rotor walk is called recurrent if it visits each vertex infinitely many times, almost surely. Otherwise is called transient. The role of the underlying graph $G$ and its properties influences the behaviour of the rotor walks on them, and this aspect is observed in the current work where the underlying state space is a non-transitive self-similar graph, the {\it infinite Sierpi\'nski gasket graph}.  In particular, we prove that a random rotor walk on the infinite gasket is recurrent, a property shared by the simple random walk as well. Recurrence and transience of rotor walks with random initial configuration has been investigated on regular trees in \cite{holroyd-propp-trees}, on Galton-Watson trees and periodic trees in \cite{ecaterina-wilfried-sebastian,wilfried-ecaterina-directed-covers-2012}. For the recurrence of $p$-walks, a process that interpolates between random and rotor walks, see also \cite{p-walks}; a generalization of $p$-walks on higher dimensional lattices  has been considered in \cite{high-dim-p-walk}, where the author investigates the recurrence and transience behaviour of the process. On state spaces other than $\Z$, where the simple random walk is recurrent, the behaviour of random rotor walks is not completely understood.

{\it Stabilization of i.i.d.~sandpiles.} A sandpile on a graph $G$ is a function $\sigma:G\to\mathbb{N}_0$, where for $x\in G$, $\sigma(x)$ represents the number of grains of sand present at site $x$. The sandpile $\sigma$ is {\it stable} if for each $x$, $\sigma(x)$ is strictly less than the number of neighbours of $x$ in $G$. If at some vertex $x$ the sandpile $\sigma(x)$ has some number of grains greater or equal to the number of neighbours of $x$, then $\sigma(x)$ is {\it unstable} and {\it topples} by sending one grain of sand to each of the neighbours. The toppling at $x$ can create other unstable vertices, and we say that $\sigma$ \textit{stabilizes} if we can reach a stable sandpile configuration containing only stable vertices by toppling each vertex finitely many times, and $\sigma$ is then called stabilizable. If the heights $(\sigma(x))_{x\in G}$ are independent and identically distributed (i.i.d.) random variables, we refer to $\sigma$ as an i.i.d.~sandpile. Conditions for sandpiles at critical density\footnote{On $\mathbb{Z}^d$, the authors of \cite{fey-redig} consider the case of $2d$ particles at a site as stable, hence a shift in their results, and they prove that a necessary condition for an i.i.d.~sandpile to stabilize is $\E[\sigma(0)]\leq 2d-1$, but this condition is not sufficient for stabilization.} on $\Z^2$ were investigated in \cite[Theorem 1]{hough-jerison-levine-2019}, where it is shown that an i.i.d.~sandpile with  $\E[\sigma(0)]$ slightly less than $3$ cannot stabilize almost surely unless $\sigma(x)\leq 3$ with high probability.
We are not aware of state spaces other than $\Z^d$ where necessary and sufficient conditions for i.i.d.~sandpiles to stabilize are given.

{\it Rotor walks and sandpiles.} While both processes may be seen as approaches to distribute chips and move particles on a graph, another relation between them may not be obvious at first sight. Indeed, there is another natural relation between these two processes, in terms of group actions. In particular, for any finite graph one can define a rotor-router group with elements being the set of acyclic rotor configurations, where a configuration is called acyclic if the rotors do not form a directed cycle. On the same graph, over the set of stable sandpile configurations, one can define in a natural way a Markov chain, by adding one chip uniformly at random and stabilizing. The set of recurrent states for this Markov chain is a group, with group operation given by pointwise addition followed by stabilization. This group is called {\it the sandpile group} or the {\it critical group} and it acts transitively on the rotor-router group. These two groups are also isomorphic; we refer to \cite{chip-firing-survey} and the references there for a beautiful exposition and more details in this direction. 

\textbf{Our contribution.} We consider rotor walks and i.i.d.~sandpiles $\sigma$ on the doubly infinite Sierpi\'nski gasket graph $\SG$ with fixed vertex $o=(0,0)$ as in Figure \ref{fig:gasket}. Our motivation for looking at such state spaces comes from physics, because Abelian sandpiles on Sierpi\'nski gasket graphs have been considered by physicists for more than 20 years ago in \cite{waves-sandpile-daerden,crit-exp-daerden-1998,KutnjakUrbanc1996SandpileMO}, where several predictions and conjectures have been made. While the conjectures are still lacking mathematical proofs, there has been some recent progress on the limit shape for the Abelian sandpile on $\SG$ in \cite{chen-sandpile-2020}. For recent results on the identity element of the sandpile group and bounds on the speed of convergence to stationarity of the Abelian sandpile Markov chain on $\SG$ see \cite{robin-yuwen-ecaterina-gasket}; for the scaling limit of the identity element, see also \cite{scaling_identity}.

We denote by $(R_t)_{t\in\mathbb{N}}$ the rotor  walk  with random initial configuration of rotors on the doubly-infinite Sierpi\'nski gasket graph $\SG$. If at the beginning of the process, for each $x\in\SG$, the rotor at  $x$ is uniformly distributed on the neighbours of $x$, that is, it points to each of the neighbours with the same probability, then we call $(R_t)_{t\in\N}$ {\it uniform rotor walk}, shortly $\mathsf{URW}$. It is not known if the uniform rotor walk on $\Z^2$ is recurrent. We prove the following.

\begin{theorem}\label{thm:rec}
The uniform rotor walk $(R_t)_{t\in\mathbb{N}}$ on the doubly-infinite Sierpi\'nski gasket graph $\SG$, with anticlockwise ordering of the neighbours of each vertex, is recurrent. 
\end{theorem}
Concerning Theorem \ref{thm:rec}, almost sure convergence of the uniform rotor walk is the best possible result one could hope for, since transient rotor configurations always exist. On the Sierpi\'nski gasket one can consider a rotor configuration which sends the particle always to the right. This configuration is a null set if rotors are chosen uniformly at random.


\begin{theorem}\label{thm:iid-sandpile}
Let $\sigma=(\sigma(x))_ {x\in \SG}$ be an i.i.d.~sandpile on $\SG$ with $\mathbb{E}[\sigma(o)]\geq 3$. If  $0<\text{Var}[\sigma(o)]<\infty$, then $\sigma$ does not stabilize almost surely.
\end{theorem}

Above, if $\mathsf{Var}[\sigma(o)]=0$, then $\sigma$ is the constant configuration which is already stable in the case $\mathbb{E}[\sigma(o)]=3$ and does not stabilize if $\mathbb{E}[\sigma(o)]>3$. 
The proof of Theorem \ref{thm:iid-sandpile} carries over to divisible sandpiles and as a consequence we obtain in Proposition \ref{prop:div} that  an i.i.d.~divisible sandpile $\sigma$ on $\SG$ with $\mathbb{E}[\sigma(o)]\geq 1$ and $0<\mathsf{Var}[\sigma(o)]<\infty$ does not stabilize alsmot surely.
We would like to emphasize here that on $\Z^d$, for an i.i.d.~sandpile with $\mathbb{E}[\sigma(0)]=2d$ both cases - stabilization and non-stabilization - can occur \cite{fey-redig}.

\section{Preliminaries and notation}

\textbf{Graphs.} Let $G=(V,E)$ be a locally finite, undirected, infinite graph with vertex set $V$ and edge set $E\subset V\times V$, and we fix a vertex $o\in V$ where particles start their walk (random or deterministic). For an edge $e=(x,y)$ we write sometimes $x\sim y$ to denote that vertices $x,y$ are neigbours. For a subset $S\subset V$, we denote by $\partial_oS$ the outer boundary of $S$, that is,
$$\partial_o S=\{x\notin S:\ x \text{ has a neighbour in } S\}.$$
We denote by $\mathsf{deg}(x)$ the degree of $x$, i.e.~the number of vertices in $G$ that are neighbours of $x$, and by abuse of notation we use $x\in G$ to denote vertices. The graph $G$ comes with a natural metric $d(x,y)$, {\it the graph distance}, i.e.~the minimal length of a path between two vertices $x$ and $y$. 

\textbf{Rotor walks.} 
For any $x\in G$, we fix an {\it ordered family of its neighbours} $\mathsf{cyc}(x)=\{x^1,\ldots,x^{\mathsf{deg}(x)}\}$, which may be thought of as the order in which, a particle exiting $x$, visits its neighbours.
For simplicity of notation, we fix during this work the {\it anticlockwise ordering of the neighbours} for each vertex. In addition, at the beginning of the process, each vertex $x$ is equipped with an  arrow (or rotor) pointing to one of the neighbours, that indicates where should a particle leaving $x$ for the first time move to. Subsequent exits from $x$ are then determined by $\mathsf{cyc}(x)$. The set $\rho$ of all rotors on $G$ is called \emph{rotor configuration}. We represent the rotor configuration $\rho$ by a function $\rho: G\to \mathbb{N}_0$, with $\rho(x)=i\in \{1,\ldots,\mathsf{deg}(x)\}$
meaning that the rotor at $x$ points to the $i$-th neighbour $x^{i}$ in the cyclic ordering $\mathsf{cyc}(x)$ of $x$. With the rotor configuration $\rho$ and the anticlockwise ordering of the neighbours for all vertices in $G$, we define recursively a {\it rotor walk $(R_t)_{t\in\mathbb{N}}$} as the sequence of consecutive locations in $G$ of a particle initially located at some fixed vertex $o\in G$ so $R_0=o$, while the subsequent locations are determined by the following rule. For any $t\geq 1$, if the location at time $t$ is $x\in G$, so $R_t=x$ then in order to determine the next move, the particle (the walker) first increments the rotor at $x$, 
 i.e.~it changes its direction to the next neighbor in the counterclockwise order $\mathsf{cyc}(x)$, and then it follows this new direction. Thus at each time step $t$, we change not only the position of the particle, but also the rotor configuration only at the current location of the particle. In other words, one starts with $(R_0,\rho_0)=(o,\rho)$ for a fixed initial rotor configuration $\rho$,  and if at time $t$ the pair (position, configuration) is $(R_t,\rho_t)$, then at time $t+1$, $(R_{t+1},\rho_{n+1})$ is
\begin{equation*}
\rho_{t+1}(x)=
\begin{cases}
\rho_t(x)+1 \mod \mathsf{deg}(x)&, \text{ if } x= R_t\\
\rho_t(x) &, \text{ otherwise}
\end{cases}
\end{equation*}
and $R_{t+1}=R_t^{\rho_{t+1}(R_t)}$, respectively.
As described here, $(R_t)_{t\in\mathbb{N}}$ is only a deterministic process, where at each step, the walker follows the prescribed rule determined by the initial rotor configuration $\rho$ and by the cyclic ordering of the neighbours. While the deterministic rotor walk $(R_t)_{t\in \mathbb{N}}$ has itself peculiar properties that are far away from being mathematically understood, in this note we focus on rotor walks with random initial configurations.

\textbf{Uniform rotor walks.} For sake of simplicity, suppose that the underlying graph $G$ is regular, so all vertices have the same degree $\mathsf{d}\in\mathbb{N}$. A {\it random initial rotor configuration} is given by a sequence
$(\rho(x))_{x\in G}$ of independent and identically distributed (i.i.d.) random variables over $\{1,\ldots,\mathsf{d}\}$. We write $\mathbb{P}$ and $\mathbb{E}$ for probability and expectation, respectively. By using the fixed anticlockwise ordering of the neighbours for any vertex and taking as the initial rotor configuration $(\rho(x))_{x\in V}$ a sequence of i.i.d.~$\{1,\ldots,\mathsf{d}\}$- valued random variables, we then obtain a rotor walk moving according to the rotor walk rule (increment rotor, move where the rotor points to), which is a random process called {\it random rotor walk}. If $(\rho(x))_{x\in G}$ is a sequence of i.i.d.~random variables uniformly distributed on $\{1,\ldots,\mathsf{d}\}$, then $(R_t)_{t\in\mathbb{N}}$ is called {\it uniform rotor walk}, (shortly $\mathsf{URW}$). A random rotor walk is not Markovian, since for $R_{t}=x$, in order to determine the position $R_{t+1}$ of the walk at time $t+1$, we need to  know where did the particle, previously exiting $x$, move to. Once a vertex has been already visited, the rotor configuration is not random anymore, and the exits from that vertex are completely deterministic.


\textbf{Stabilization of i.i.d.~sandpiles.} There are several ways to stabilize an unstable sandpile configuration $\sigma$ on infinite graphs $G$; see \cite{fey-redig} for more details on toppling procedures and stabilizing configurations. While on $\Z^d$ the stabilization proofs rely on translation invariance, ergodic theory and potential kernel estimates, such techniques cannot immediately be applied to graphs that are not transitive.
For any sequence $(G_n)_{n\in\mathbb{N}}$ of finite subsets $G_n\subset G_{n+1}\subset G$ exhausting $G$, with $G=\cup_{n}G_n$, such that for each $n\in\mathbb{N}$, $G_n$ has a global sink $s$ given by the boundary vertices of $G_n$, the sandpile $\sigma_n$ - where $\sigma_n$ is the restriction of the globally given configuration $\sigma$ on the finite set $G_n$ of the exhaustion - stabilizes in finitely many steps and the sink collects the excess mass. We denote by $\sigma_n^{\infty}$ the stabilization of $\sigma_n$ and by $u_n(x)$ the mass emitted from $x$ during stabilization; the function $u_n(x)$ is called {\it the odometer function}. We have that
$\sigma_n^{\infty}=\sigma_n-\Delta u_n$,
where $\Delta$  is {\it  the graph Laplacian} defined as
$$\Delta u_n(x)=u_n(x)-\frac{1}{\mathsf{deg}(x)}\sum_{y\sim x}u_n(y).$$
{\it Toppling in infinite volume.} For an i.i.d.~sandpile configuration $\sigma$ on $G$, and for the sequence of increasing graphs $(G_n)$ that exhaust $G$, toppling in infinite volume reduces to first stabilizing $\sigma$ in $G_1$, then the new sandpile configuration is being stabilized in $G_2$, and so on. Since for any $n\in\N$, $G_n$ is finite, the sequence $(u_n)$ of odometer functions on $G$ is well defined. 

We call the i.i.d.~sandpile $\sigma$ \textit{stabilizable} (in infinite volume) if it exists a function $u^{\infty}:G\to\mathbb{R}$ such that $\lim_{n\to\infty}u_n=u^{\infty}$ pointwise,
where $u_n$ is the odometer function for the stabilization in $G_n$. The final stable configuration $\sigma^{\infty}:G\to\mathbb{R}$ is given by
$\sigma^{\infty}=\sigma-\Delta u^{\infty}$.
Note that the odometer functions are monotonically increasing, so the limit $\lim_{n\to\infty}u_n$ always exists, but it
may be infinite at some vertices.
Therefore $\sigma$ is not stabilzable if there exists $x\in G$ such that $\lim_{n\to\infty}u_n(x)=\infty$. If 
$\sigma$ is not stabilizable, we also call $\sigma$ an {\it exploding sandpile}.

\begin{remark}\label{rem:monotonicity_stab}
For any sandpile configuration $\sigma$ on a finite graph $G$, the corresponding odometer function can also be characterized as the smallest, non-negative function $u:G\to\mathbb{Z}$ that satisfies
$\sigma(x)-\Delta u(x) \leq \mathsf{deg}(x)-1$
for any vertex $x\in G$. Thus we infer that for two i.i.d.~sandpiles $\sigma_1,\sigma_2$ with $\sigma_1\leq \sigma_2$ a.s., if $\sigma_2$ stabilizes almost surely, so does $\sigma_1$. 
\end{remark}

\textbf{Sierpi\'nski gasket graph $\SG$} is a pre-fractal associated with the Sierpi\'nski gasket, defined as following. We consider in $\mathbb{R}^2$ the sets
$V_0=\{(0,0), (1,0), (1/2,\sqrt{3}/2)\}$
and 
\begin{equation*}
E_0=\left\{\big((0,0),(1,0)\big),\big((0,0),(1/2,\sqrt{3}/2)\big),\big((1,0),(1/2,\sqrt{3}/2)\big)\right\}.
\end{equation*}
Now recursively define $(V_1,E_1), (V_2,E_2),\ldots$ by
\begin{equation*}
V_{n+1}=V_n\cup\left\{\big(2^n,0\big)+V_n\right\}\bigcup \left\{\left(2^{n-1},2^{n-1}\sqrt{3}\right)+V_n\right\}
\end{equation*}
and 
\begin{equation*}
E_{n+1}=E_n\cup\left\{\left(2^n,0\right)+E_n\right\}\bigcup \left\{\left(2^{n-1},2^{n-1}\sqrt{3}\right)+E_n\right\},
\end{equation*}
where $(x,y)+S:=\{(x,y)+s:s\in S\}$. 
 Let $V_{\infty}=\cup_{n=0}^{\infty}V_n$ and $V_{\infty}'=\{(-x,y):\ (x,y)\in V_{\infty}\}$ be the reflection of $V_{\infty}$ around the vertical axis. Similarly, let $E_{\infty}=\cup_{n=0}^{\infty}E_n$ and denote the reflection of $E_{\infty}$ around the vertical axis by $E'_{\infty}$. Finally, let 
 $V=V_{\infty}\cup V'_{\infty}$ and $E=E_{\infty}\cup E'_{\infty}$. Then the doubly infinite Sierpi\'nski gasket graph $\SG$ (called also {\it graphical Sierpi\'nski gasket})
is the graph with vertex set $V$ and edge set $E$. The associated level-$n$ prefractal graphs $\SG_n$ are the graphs with vertex set $V_n\cup V_n'$ and edge set $E_n\cup E_n'$, where as above $V_n'$ and $E_n'$ are the reflections around the $y$-axis of $V_n$ and $E_n$ respectively. So we can write $\SG=\cup_{n\in\mathbb{N}}\SG_n$.
See Figure \ref{fig:gasket} for a graphical representation of $\SG$. Denote by $\SG^{+}$ (respectively  $\SG_n^{+}$) the graph with vertex set $V_{\infty}$ (respectively $V_n$) and edge set $E_{\infty}$ (respectively), and similarly $\SG^{-}$ and $\SG_n^{-}$ for their respective reflections around the $y$-axis, so we can finally write $\SG=\SG^{+}\cup\SG^{-}$ and $\SG_n=\SG_n^{+}\cup\SG_n^{-}$. 
Set the origin $o=(0,0)$.

\begin{small}
\begin{figure}
\label{fig:gasket}
	\centering
	\input{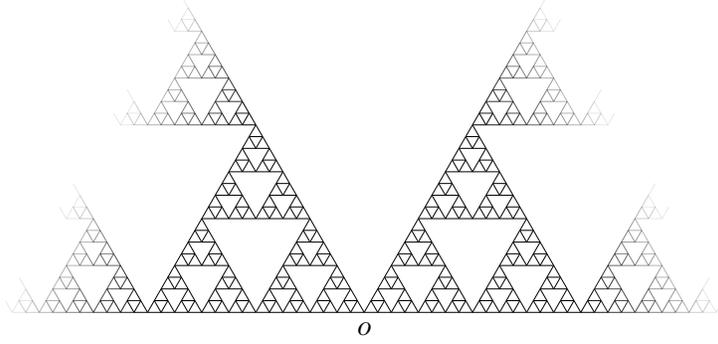}
	\caption{\label{fig:gasket} Doubly-infinite Sierpinski gasket graph $\SG$.}
\end{figure}
\end{small}

\section{Uniform rotor walks on the Sierpi\'nski gasket}

For the rest we fix $\SG$ to be the doubly-infinite Sierpi\'nski gasket graph with root vertex (lower corner) $o=(0,0)$, and $\SG_n$, $n\in \mathbb{N}$ the level-$n$ prefractals approximating $\SG$ with $\SG=\cup_{n\in\mathbb{N}}\SG_n$. All the vertices in $\SG$ have degree $4$, and for any $x\in \SG$ we fix the anticlockwise ordering of the neighbours, and the initial random configuration of rotors uniformly distributed over the set $\{1,2,3,4\}$, that is, the rotor at $x$ points initially to each of the four possible neighbours with probability $\frac{1}{4}$.

For any initial random rotor configuration $\rho$ on $\SG$, any rotor mechanism, and any starting position, one can see that the random rotor walk either visits each vertex infinitely many times almost surely, and we call this walk {\it recurrent}, or visits each vertex only finitely many times almost surely, and we call the walk {\it transient}. In particular, if the uniform rotor walk started at the origin $o\in \SG$ returns to the origin infinitely many times almost surely, then the uniform rotor walk is recurrent. We prove recurrence of $(R_t)_{t\in\mathbb{N}}$ on $\SG$ by using {\it sets with reflecting boundary} as introduced in \cite{holroyd-angel-rec-cfg}.

\begin{definition}
For a given rotor mechanism  (i.e.~cyclic configuration) and an initial rotor configuration $\rho$ on $\SG$, we say that a subset $S$ of vertices in $\SG$ has \emph{ reflecting boundary} if for every vertex $y$ in the outer boundary $\partial_oS$, the rotor at $y$ will first send the particle to each of the $y$'s neigbours in $S$, before sending it to any other neighbour of $y$. 
\end{definition}

The self-similar nature of $\SG$ together with the existence of the cut points that disconnect the gasket into finitely many connected components, allows one to prove the existence of infinitely many sets with reflecting boundary, almost surely, thus proving Theorem \ref{thm:rec}. We recall first \cite[Proposition 7]{holroyd-angel-rec-cfg},  which claims that if for some rotor configuration $\rho$, every finite set of vertices is a subset of some finite set with reflecting boundary, then the rotor walk with rotor configuration $\rho$ starting from any other vertex is recurrent. Additionally, we state also       
 \cite[Lemma 9]{holroyd-angel-rec-cfg} which will be used in the proof below: if $S$ is a set with reflecting boundary, then the rotor walk started at $x\in S$ will return to $x$ before leaving $S\cup\partial_o S$. 
 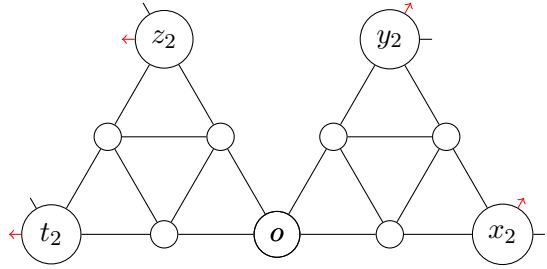
\begin{figure}  
        \centering
        \begin{tikzpicture}[scale=0.75]
        \node[shape=circle,draw=black] (A) at (0,0) {$o$};
        \node[shape=circle,draw=black] (B) at (2,0) {};
        \node[shape=circle,draw=black] (C) at (4,0) {$x_2$};
        \node[shape=circle,draw=black] (D) at (1,1.73) {};
        \node[shape=circle,draw=black] (E) at (3,1.73) {};
        \node[shape=circle,draw=black] (F) at (2,1.73*2) {$y_2$} ;
        
        \node[shape=circle,draw=none] (A1) at (-1,0) {};
        \node[shape=circle,draw=none] (A2) at (-1/2,-1.73/2) {};
        \node[shape=circle,draw=none] (B1) at (5,0) {};
        \node[shape=circle,draw=none] (B2) at (4+1/2,1.73/2) {};
        \node[shape=circle,draw=none] (C1) at (3,1.73*2) {};
        \node[shape=circle,draw=none] (C2) at (5/2,1.73*2+1.73/2) {};
        
        \path [-] (A) edge node[left] {} (B);
        \path [-] (A) edge node[left] {} (D);
        \path [-] (B) edge node[left] {} (D);
        \path [-] (B) edge node[left] {} (C);
        \path [-] (E) edge node[left] {} (B);
        \path [-] (E) edge node[left] {} (C);
        \path [-] (F) edge node[left] {} (E);
        \path [-] (F) edge node[left] {} (D);
        \path [-] (E) edge node[left] {} (D);
        
        \path [-] (C) edge node[left] {} (B1);
        \path [->, draw = red] (C) edge node[left] {} (B2);
        \path [-] (F) edge node[left] {} (C1);
        \path [->, draw=red] (F) edge node[left] {} (C2);
        
        \node[shape=circle,draw=black] (-A) at (0,0) {$o$};
        \node[shape=circle,draw=black] (-B) at (-2,0) {};
        \node[shape=circle,draw=black] (-C) at (-4,0) {$t_2$};
        \node[shape=circle,draw=black] (-D) at (-1,1.73) {};
        \node[shape=circle,draw=black] (-E) at (-3,1.73) {};
        \node[shape=circle,draw=black] (-F) at (-2,1.73*2) {$z_2$} ;
        
        \node[shape=circle,draw=none] (-B1) at (-5,0) {};
        \node[shape=circle,draw=none] (-B2) at (-4-1/2,1.73/2) {};
        \node[shape=circle,draw=none] (-C1) at (-3,1.73*2) {};
        \node[shape=circle,draw=none] (-C2) at (-5/2,1.73*2+1.73/2) {};
        
        \path [-] (-A) edge node[left] {} (-B);
        \path [-] (-A) edge node[left] {} (-D);
        \path [-] (-B) edge node[left] {} (-D);
        \path [-] (-B) edge node[left] {} (-C);
        \path [-] (-E) edge node[left] {} (-B);
        \path [-] (-E) edge node[left] {} (-C);
        \path [-] (-F) edge node[left] {} (-E);
        \path [-] (-F) edge node[left] {} (-D);
        \path [-] (-E) edge node[left] {} (-D);
        
        \path [->, draw=red] (-C) edge node[left] {} (-B1);
        \path [-] (-C) edge node[left] {} (-B2);
        \path [->, draw=red] (-F) edge node[left] {} (-C1);
        \path [-] (-F) edge node[left] {} (-C2);
        \end{tikzpicture}
         \caption{The sets $SG_2$ and $S_2=SG_2\setminus\{x_2,y_2,z_2,t_2\}$ together with the rotors at the outer boundary of $S_2$ such that $S_2$ has reflecting boundary.}
    \label{fig:sg2}
    \end{figure}
\begin{proof}[Proof of Theorem \ref{thm:rec}]
For any $n\in\mathbb{N}$ and $\SG_n\subset \SG$, we define $S_n=SG_n\setminus\{x_n,y_n,z_n,t_n\}$ where the vertices $\{x_n,y_n,z_n,t_n\}$ represent the four inner boundary vertices of $\SG_n$, i.e.~the set of vertices of $\SG_n$ that have a neighbour in $\SG\setminus \SG_n$. Using the definition of $\SG$, we can represent these four vertices by their coordinates in $\mathbb{R}^2$ as $x_n=(2^n,0)$, $y_n=(2^{n-1},2^{n-1}\sqrt{3})$,
$z_n=(-2^{n-1},2^{n-1}\sqrt{3})$ and $t_n=(-2^n,0)$. So the outer boundary of $S_n$ is 
$\partial_oS_n=\{x_n,y_n,z_n,t_n\}$.
Consider now the uniform rotor configuration on $\SG$, given by a sequence $(\rho(x))_{x\in \SG}$ of i.i.d.~random variables, such that for each $x\in\SG$, $\rho(x)$ points with probability $1/4$ to each of the four neighbours, and suppose that the corresponding rotor walk starts at $o=(0,0)$. Recall that we have assumed for any $x\in \SG$ an anticlockwise ordering of the neighbours. Then, for the uniform random configuration $\rho$ and anticlockwise ordering of the neighbours, for each $y\in\partial_oS_n$, with probability $1/4$, the rotor at $y$ points to a neighbour that will send the particle first to each neighbours in $S_n$ before sending it to the other neighbours. Since the random variables $(\rho(y))_{y\in \partial_oS_n}$ are i.i.d., this implies that $S_n$ has a reflecting boundary with probability $\left(\frac{1}{4}\right)^4=\frac{1}{256}$,  for all $n\in\mathbb{N}$. See Figure  \ref{fig:sg2} and Figure  \ref{fig:sg3} for the sets $S_2$ and $S_3$ respectively, and the corresponding rotor configurations at the outer boundary points for which $S_2$ and $S_3$ are sets with reflecting boundary. Consider now the sequence of events $(A_n)_{n\in\mathbb{N}}$ given by
$$A_n=\{S_n \text{ has reflecting boundary}\},$$
which, by the definition of $S_n$ and of the uniform rotor walk, are independent. Since $\sum_{n\in\mathbb{N}}\mathbb{P}(A_n)=\infty$, Borel-Cantelli Lemma implies that infinitely many of the events $A_n$ occur, almost surely. Therefore, with probability one, there are infinitely many subsets $S_n$ of $\SG$ with reflecting boundary.

Now the recurrence of the uniform rotor walk on $\SG$ follows from \cite[Lemma 9]{holroyd-angel-rec-cfg}. Due to the construction of the nested sets $S_n$ (i.e for $n_1<n_2$, we have $S_{n_1}\subset S_{n_2}$) with reflecting boundary, since each of the $S_n$ contains the origin, and there are infinitely many sets $S_n$ with reflecting boundary almost surely, we thus obtain that the origin is visited infinitely many times by the uniform rotor walk, almost surely, and this proves recurrence of the rotor walk started at the origin.
If the uniform rotor walk starts at another vertex $x$ different than the origin, the recurrence follows from  \cite[Proposition 7]{holroyd-angel-rec-cfg}, since with our choice of the sets $S_n$, for any finite subset $B\subset \SG$, we can find an index big enough $n'$ such that $S_{n'}$ has reflecting boundary and $B\subset S_{n'}$ almost surely. This yields the recurrence of the uniform rotor walk for any other starting point.
\end{proof}
We would like to point out  that in the proof above we have considered uniform rotor walk, but any other random initial configuration of rotors $(\rho(x))_{x\in\SG}$ that points to one of the four possible neighbours with probability $p_i$, $i\in\{1,2,3,4\}$ and fullfils $(\min\{p_1,p_2,p_3,p_4\})^4=c>0$, works as well in the proof.

\begin{proposition}
A random rotor walk $(R_t)_{t\in\mathbb{N}}$ with i.i.d.~random initial configuration $(\rho(x))_{x\in\SG}$ such that for each $x\in\SG$, $\rho(x)$ points to each of the neighbours $x^i$ of $x$ with positive probability $p_i>0$ for $i\in\{1,2,3,4\}$, is recurrent.
\end{proposition}
The proof is exactly the same as the one of Theorem \ref{thm:rec}, with the minor difference that the sets $S_n$ have reflecting boundary with probability $(\min\{p_1,p_2,p_3,p_4\})^4=c>0$, which is a condition generally known as the elliptic assumption.

\begin{figure}
       \centering
        \begin{tikzpicture}[scale=0.9]
        \node[shape=circle,draw=black] (A) at (0,0) {$o$};
        \node[shape=circle,draw=black] (B) at (2,0) {};
        \node[shape=circle,draw=black] (C) at (4,0) {$x_3$};
        \node[shape=circle,draw=black] (D) at (1,1.73) {};
        \node[shape=circle,draw=black] (E) at (3,1.73) {};
        \node[shape=circle,draw=black] (F) at (2,1.73*2) {$y_3$} ;
        \node[shape=circle,draw=black] (G) at (1,0) {};
        \node[shape=circle,draw=black] (H) at (3,0) {};
        \node[shape=circle,draw=black] (I) at (1/2,1.73/2) {};
        \node[shape=circle,draw=black] (J) at (3/2,1.73/2) {};
        \node[shape=circle,draw=black] (K) at (5/2,1.73/2) {};
        \node[shape=circle,draw=black] (L) at (7/2,1.73/2) {};
        \node[shape=circle,draw=black] (M) at (3/2,1.73/2+1.73) {};
        \node[shape=circle,draw=black] (N) at (5/2,1.73/2+1.73) {};
        \node[shape=circle,draw=black] (O) at (2,1.73) {};
        
        \node[shape=circle,draw=none] (A1) at (-1,0) {};
        \node[shape=circle,draw=none] (A2) at (-1/2,-1.73/2) {};
        \node[shape=circle,draw=none] (B1) at (5,0) {};
        \node[shape=circle,draw=none] (B2) at (4+1/2,-1.73/2) {};
        \node[shape=circle,draw=none] (C1) at (3,1.73*2) {};
        \node[shape=circle,draw=none] (C2) at (5/2,1.73*2+1.73/2) {};
        
        \node[shape=circle,draw=none] (X) at (4+1/2,1.73/2) {};
        \path[->, draw=red] (C) edge node[left] {} (X);
        
        \node[shape=circle, draw=none] (Y) at (2+1/2,2*1.73+1.73/2) {};
        \path[->, draw=red] (F) edge node[left] {} (Y);
        
        \path[-] (C) edge node[left] {} (B1);
        \path[-] (F) edge node[left] {} (C1);
        
        \path [-] (A) edge node[left] {} (G);
        \path [-] (B) edge node[left] {} (G);
        \path [-] (A) edge node[left] {} (I);
        \path [-] (D) edge node[left] {} (I);
        \path [-] (B) edge node[left] {} (J);
        \path [-] (D) edge node[left] {} (J);
        \path [-] (I) edge node[left] {} (J);
        \path [-] (I) edge node[left] {} (G);
        \path [-] (J) edge node[left] {} (G);
        
        \path [-] (B) edge node[left] {} (H);
        \path [-] (C) edge node[left] {} (H);
        \path [-] (E) edge node[left] {} (K);
        \path [-] (E) edge node[left] {} (L);
        \path [-] (B) edge node[left] {} (K);
        \path [-] (C) edge node[left] {} (L);
        \path [-] (H) edge node[left] {} (K);
        \path [-] (H) edge node[left] {} (L);
        \path [-] (L) edge node[left] {} (K);
        
        \path [-] (F) edge node[left] {} (M);
        \path [-] (F) edge node[left] {} (N);
        \path [-] (E) edge node[left] {} (N);
        \path [-] (E) edge node[left] {} (O);
        \path [-] (D) edge node[left] {} (O);
        \path [-] (D) edge node[left] {} (M);
        \path [-] (M) edge node[left] {} (N);
        \path [-] (O) edge node[left] {} (N);
        \path [-] (O) edge node[left] {} (M);
        
        \node[shape=circle,draw=black] (A-) at (0,0) {$o$};
        \node[shape=circle,draw=black] (B-) at (-2,0) {};
        \node[shape=circle,draw=black] (C-) at (-4,0) {$t_3$};
        \node[shape=circle,draw=black] (D-) at (-1,1.73) {};
        \node[shape=circle,draw=black] (E-) at (-3,1.73) {};
        \node[shape=circle,draw=black] (F-) at (-2,1.73*2) {$z_3$} ;
        \node[shape=circle,draw=black] (G-) at (-1,0) {};
        \node[shape=circle,draw=black] (H-) at (-3,0) {};
        \node[shape=circle,draw=black] (I-) at (-1/2,1.73/2) {};
        \node[shape=circle,draw=black] (J-) at (-3/2,1.73/2) {};
        \node[shape=circle,draw=black] (K-) at (-5/2,1.73/2) {};
        \node[shape=circle,draw=black] (L-) at (-7/2,1.73/2) {};
        \node[shape=circle,draw=black] (M-) at (-3/2,1.73/2+1.73) {};
        \node[shape=circle,draw=black] (N-) at (-5/2,1.73/2+1.73) {};
        \node[shape=circle,draw=black] (O-) at (-2,1.73) {};
        
        \node[shape=circle,draw=none] (B1-) at (-5,0) {};

        \node[shape=circle,draw=none] (C1-) at (-3,1.73*2) {};

        \node[shape=circle,draw=none] (X-) at (-4-1/2,1.73/2) {};
        \path[-] (C-) edge node[left] {} (X-);
        
        \node[shape=circle, draw=none] (Y-) at (-2-1/2,2*1.73+1.73/2) {};
        \path[-] (F-) edge node[left] {} (Y-);
        
        \path[->, draw=red] (C-) edge node[left] {} (B1-);
        \path[->, draw=red] (F-) edge node[left] {} (C1-);
        
        \path [-] (A-) edge node[left] {} (G-);
        \path [-] (B-) edge node[left] {} (G-);
        \path [-] (A-) edge node[left] {} (I-);
        \path [-] (D-) edge node[left] {} (I-);
        \path [-] (B-) edge node[left] {} (J-);
        \path [-] (D-) edge node[left] {} (J-);
        \path [-] (I-) edge node[left] {} (J-);
        \path [-] (I-) edge node[left] {} (G-);
        \path [-] (J-) edge node[left] {} (G-);
        
        \path [-] (B-) edge node[left] {} (H-);
        \path [-] (C-) edge node[left] {} (H-);
        \path [-] (E-) edge node[left] {} (K-);
        \path [-] (E-) edge node[left] {} (L-);
        \path [-] (B-) edge node[left] {} (K-);
        \path [-] (C-) edge node[left] {} (L-);
        \path [-] (H-) edge node[left] {} (K-);
        \path [-] (H-) edge node[left] {} (L-);
        \path [-] (L-) edge node[left] {} (K-);
        
        \path [-] (F-) edge node[left] {} (M-);
        \path [-] (F-) edge node[left] {} (N-);
        \path [-] (E-) edge node[left] {} (N-);
        \path [-] (E-) edge node[left] {} (O-);
        \path [-] (D-) edge node[left] {} (O-);
        \path [-] (D-) edge node[left] {} (M-);
        \path [-] (M-) edge node[left] {} (N-);
        \path [-] (O-) edge node[left] {} (N-);
        \path [-] (O-) edge node[left] {} (M-);
        \end{tikzpicture}
         \caption{The sets $SG_3$ and $S_3=SG_3\setminus\{x_3,y_3,z_3,t_3\}$ together with the rotors at the outer boundary of $S_3$ such that $S_3$ has reflecting boundary.}
          \label{fig:sg3}
        \end{figure}
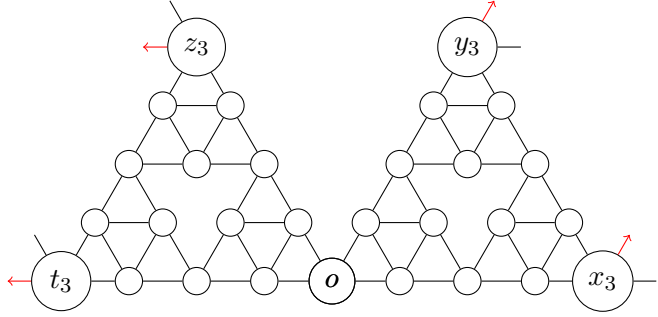

\section{Stabilization of i.i.d.~sandpiles}

Let $(\sigma(x))_{x\in\SG}$ be an i.i.d.~sandpile on $\SG$. 
For $\SG=\cup_{n\in\mathbb{N}}\SG_n$, and each $n\in \mathbb{N}$, we take again $\SG_n=\SG_n^{+}\cup\SG_n^{-}$ as the level $n$-gasket with sink vertices $\{x_n,y_n,z_n,t_n\}$ (the four corners). Similar to \cite[Theorem 3.5]{fey-meester-redig-2009}, we give a necessary condition for an i.i.d.~sandpile on the gasket $\SG$ to stabilize. 
Our method explores the existence of cut points in the gasket, so in order for the excess mass to pass from one iteration $\SG_n$ to the next one $\SG_{n+1}$ of the gasket $\SG$, it has to pass through the cut points. Together with the Abelian property of the sandpile model, this enables to run the sandpile in waves.
From now on,  we assume that all our sandpiles have finite variance of the number of chips, 
$\sigma_0^2:=\mathsf{Var}[\sigma(o)]<\infty$.

For the {\it toppling in infinite volume} on $\SG$, we use as the sequence of increasing subsets of $\SG$ the level-$n$ prefractals $\SG_n$, with $\SG=\cup_{n\in\N}\SG_n$, and for every $n\in \N$ we also use the partition of $\SG_n$ in the right triangle $\SG_n^{+}$ with three boundary vertices $\{o,x_n,y_n\}$ and the left triangle $\SG_n^{-}$ with three boundary vertices $\{o,z_n,t_n\}$. As sinks we use the outer boundary $\partial_o \SG_n$. For an i.i.d.~sandpile  $\sigma$ on $\SG$, in order to stabilize $\sigma$ on $\SG_n$, by the Abelian property of the model, we can stabilize it first on $\SG_n^{+}$, and investigate the mass that has to be emitted from the origin $o$ in the course of the stabilization. Then, in order to stabilize the sandpile on the left triangle $\SG_n^{-}$, the odometer at $o$ can only increase. So, bounding from below the odometer at $o$ during the stabilization of $\sigma$ on $\SG_n^{+}$ by a function that goes to $\infty$ with $n$, and showing that such an event occurs with positive probability together with a Borel-Cantelli type argument proves Theorem \ref{thm:iid-sandpile}. 
First of all, for any $m,n\in\mathbb{N}$ such that $m<n-1$, the graph $\SG^+$ contains three isomorphic copies of $\SG_m^+$ 
on its three corners as in Figure \ref{fig:inner_hex}. Thus, the copy $\SG_m^+$ has $o$ as bottom left corner, $x_m$ as bottom right corner and $y_m$ as top corner. The $120^{\circ}$ counterclockwise rotation of the copy $\SG_m^+$ in $\SG_n$ has $x_n$ as bottom right corner, and we denote by $a_m$ and $b_m$ the top and bottom left corner of this copy respectively. Finally the $240^{\circ}$ counterclockwise rotation of the copy $\SG_m^+$ has $y_n$ as top corner and we denote by $c_m$ and $d_m$ the bottom left and bottom right corner respectively.
Denote by $\SG_{m,n}$ the graph obtained by removing these three copies from $\SG^+$. If $m=n-1$, then $\SG_{n-1,n}$ is only a triangle, so it makes sense to assume that $m<n-1$. In this case, $\SG_{m,n}$ is a subgraph of $\SG_n^+$, and we denote by $V_{m,n}$ the vertex set of $\SG_{m,n}$. The six inner boundary vertices of $\SG_{m,n}$ by are  given by $\{x_m,y_m,a_m,b_m,c_m,d_m\}$. 

For the next three Lemmas, we assume that we are in the setting of Theorem  \ref{thm:iid-sandpile}, that is $\sigma=(\sigma(x))_{x\in\SG}$ is an i.i.d.~sandpile configuration on $\SG$ with $\mathbb{E}[\sigma(o)]\geq 3$, and we write $N_{m,n}=\sum_{x\in V_{m,n}}\sigma(x)$.

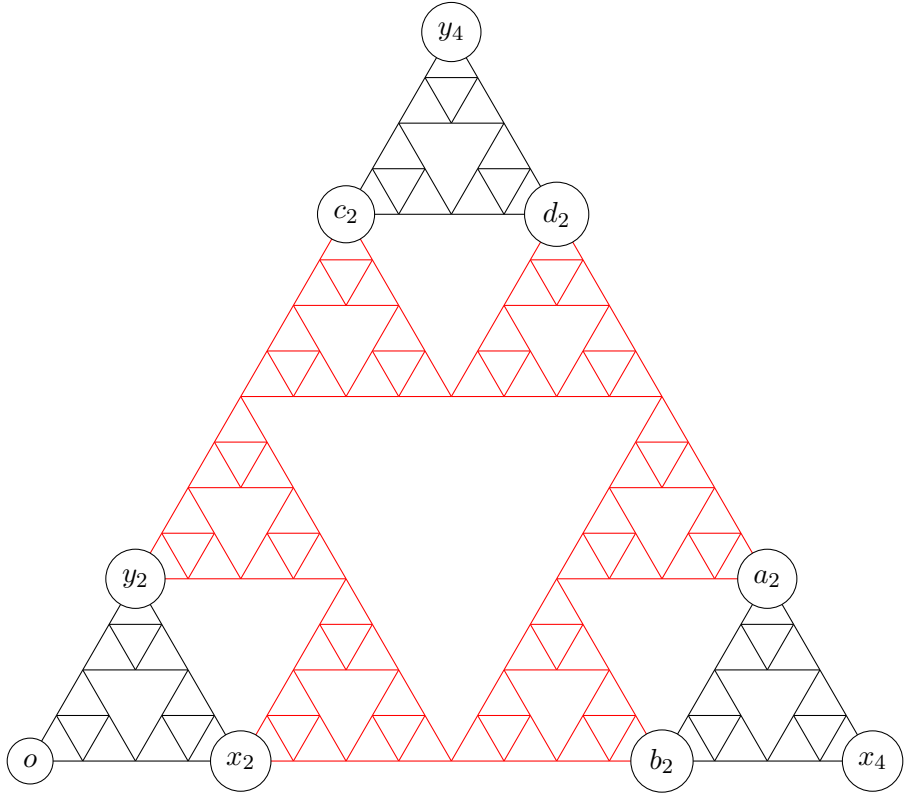
\begin{figure}
\centering
\begin{tikzpicture}[scale=0.7]
\draw[color=black] (0, 0)--(1, 0);
\draw[color=black] (1, 0)--(0.5, 0.8660254037844386);
\draw[color=black] (0, 0)--(0.5, 0.8660254037844386);
\draw[color=black] (1, 0)--(2, 0);
\draw[color=black] (0.5, 0.8660254037844386)--(1.5, 0.8660254037844386);
\draw[color=black] (2, 0)--(1.5, 0.8660254037844386);
\draw[color=black] (1.5, 0.8660254037844386)--(1.0, 1.7320508075688772);
\draw[color=black] (1, 0)--(1.5, 0.8660254037844386);
\draw[color=black] (0.5, 0.8660254037844386)--(1.0, 1.7320508075688772);
\draw[color=black] (2, 0)--(3, 0);
\draw[color=black] (1, 1.7320508075688772)--(2, 1.7320508075688772);
\draw[color=black] (3, 0)--(2.5, 0.8660254037844386);
\draw[color=black] (2, 1.7320508075688772)--(1.5, 2.598076211353316);
\draw[color=black] (2, 0)--(2.5, 0.8660254037844386);
\draw[color=black] (1, 1.7320508075688772)--(1.5, 2.598076211353316);
\draw[color=black] (3, 0)--(4, 0);
\draw[color=black] (2, 1.7320508075688772)--(3, 1.7320508075688772);
\draw[color=black] (2.5, 0.8660254037844386)--(3.5, 0.8660254037844386);
\draw[color=black] (1.5, 2.598076211353316)--(2.5, 2.598076211353316);
\draw[color=black] (4, 0)--(3.5, 0.8660254037844386);
\draw[color=black] (3, 1.7320508075688772)--(2.5, 2.598076211353316);
\draw[color=black] (3.5, 0.8660254037844386)--(3.0, 1.7320508075688772);
\draw[color=black] (2.5, 2.598076211353316)--(2.0, 3.4641016151377544);
\draw[color=black] (3, 0)--(3.5, 0.8660254037844386);
\draw[color=black] (2, 1.7320508075688772)--(2.5, 2.598076211353316);
\draw[color=black] (2.5, 0.8660254037844386)--(3.0, 1.7320508075688772);
\draw[color=black] (1.5, 2.598076211353316)--(2.0, 3.4641016151377544);
\draw[color=red] (4, 0)--(5, 0);
\draw[color=red] (2, 3.4641016151377544)--(3, 3.4641016151377544);
\draw[color=red] (5, 0)--(4.5, 0.8660254037844386);
\draw[color=red] (3, 3.4641016151377544)--(2.5, 4.330127018922193);
\draw[color=red] (4, 0)--(4.5, 0.8660254037844386);
\draw[color=red] (2, 3.4641016151377544)--(2.5, 4.330127018922193);
\draw[color=red] (5, 0)--(6, 0);
\draw[color=red] (3, 3.4641016151377544)--(4, 3.4641016151377544);
\draw[color=red] (4.5, 0.8660254037844386)--(5.5, 0.8660254037844386);
\draw[color=red] (2.5, 4.330127018922193)--(3.5, 4.330127018922193);
\draw[color=red] (6, 0)--(5.5, 0.8660254037844386);
\draw[color=red] (4, 3.4641016151377544)--(3.5, 4.330127018922193);
\draw[color=red] (5.5, 0.8660254037844386)--(5.0, 1.7320508075688772);
\draw[color=red] (3.5, 4.330127018922193)--(3.0, 5.196152422706632);
\draw[color=red] (5, 0)--(5.5, 0.8660254037844386);
\draw[color=red] (3, 3.4641016151377544)--(3.5, 4.330127018922193);
\draw[color=red] (4.5, 0.8660254037844386)--(5.0, 1.7320508075688772);
\draw[color=red] (2.5, 4.330127018922193)--(3.0, 5.196152422706632);
\draw[color=red] (6, 0)--(7, 0);
\draw[color=red] (4, 3.4641016151377544)--(5, 3.4641016151377544);
\draw[color=red] (5, 1.7320508075688772)--(6, 1.7320508075688772);
\draw[color=red] (3, 5.196152422706632)--(4, 5.196152422706632);
\draw[color=red] (7, 0)--(6.5, 0.8660254037844386);
\draw[color=red] (5, 3.4641016151377544)--(4.5, 4.330127018922193);
\draw[color=red] (6, 1.7320508075688772)--(5.5, 2.598076211353316);
\draw[color=red] (4, 5.196152422706632)--(3.5, 6.06217782649107);
\draw[color=red] (6, 0)--(6.5, 0.8660254037844386);
\draw[color=red] (4, 3.4641016151377544)--(4.5, 4.330127018922193);
\draw[color=red] (5, 1.7320508075688772)--(5.5, 2.598076211353316);
\draw[color=red] (3, 5.196152422706632)--(3.5, 6.06217782649107);
\draw[color=red] (7, 0)--(8, 0);
\draw[color=red] (5, 3.4641016151377544)--(6, 3.4641016151377544);
\draw[color=red] (6, 1.7320508075688772)--(7, 1.7320508075688772);
\draw[color=red] (4, 5.196152422706632)--(5, 5.196152422706632);
\draw[color=red] (6.5, 0.8660254037844386)--(7.5, 0.8660254037844386);
\draw[color=red] (4.5, 4.330127018922193)--(5.5, 4.330127018922193);
\draw[color=red] (5.5, 2.598076211353316)--(6.5, 2.598076211353316);
\draw[color=red] (3.5, 6.06217782649107)--(4.5, 6.06217782649107);
\draw[color=red] (8, 0)--(7.5, 0.8660254037844386);
\draw[color=red] (6, 3.4641016151377544)--(5.5, 4.330127018922193);
\draw[color=red] (7, 1.7320508075688772)--(6.5, 2.598076211353316);
\draw[color=red] (5, 5.196152422706632)--(4.5, 6.06217782649107);
\draw[color=red] (7.5, 0.8660254037844386)--(7.0, 1.7320508075688772);
\draw[color=red] (5.5, 4.330127018922193)--(5.0, 5.196152422706632);
\draw[color=red] (6.5, 2.598076211353316)--(6.0, 3.4641016151377544);
\draw[color=red] (4.5, 6.06217782649107)--(4.0, 6.928203230275509);
\draw[color=red] (7, 0)--(7.5, 0.8660254037844386);
\draw[color=red] (5, 3.4641016151377544)--(5.5, 4.330127018922193);
\draw[color=red] (6, 1.7320508075688772)--(6.5, 2.598076211353316);
\draw[color=red] (4, 5.196152422706632)--(4.5, 6.06217782649107);
\draw[color=red] (6.5, 0.8660254037844386)--(7.0, 1.7320508075688772);
\draw[color=red] (4.5, 4.330127018922193)--(5.0, 5.196152422706632);
\draw[color=red] (5.5, 2.598076211353316)--(6.0, 3.4641016151377544);
\draw[color=red] (3.5, 6.06217782649107)--(4.0, 6.928203230275509);
\draw[color=red] (8, 0)--(9, 0);
\draw[color=red] (4, 6.928203230275509)--(5, 6.928203230275509);
\draw[color=red] (9, 0)--(8.5, 0.8660254037844386);
\draw[color=red] (5, 6.928203230275509)--(4.5, 7.794228634059947);
\draw[color=red] (8, 0)--(8.5, 0.8660254037844386);
\draw[color=red] (4, 6.928203230275509)--(4.5, 7.794228634059947);
\draw[color=red] (9, 0)--(10, 0);
\draw[color=red] (5, 6.928203230275509)--(6, 6.928203230275509);
\draw[color=red] (8.5, 0.8660254037844386)--(9.5, 0.8660254037844386);
\draw[color=red] (4.5, 7.794228634059947)--(5.5, 7.794228634059947);
\draw[color=red] (10, 0)--(9.5, 0.8660254037844386);
\draw[color=red] (6, 6.928203230275509)--(5.5, 7.794228634059947);
\draw[color=red] (9.5, 0.8660254037844386)--(9.0, 1.7320508075688772);
\draw[color=red] (5.5, 7.794228634059947)--(5.0, 8.660254037844386);
\draw[color=red] (9, 0)--(9.5, 0.8660254037844386);
\draw[color=red] (5, 6.928203230275509)--(5.5, 7.794228634059947);
\draw[color=red] (8.5, 0.8660254037844386)--(9.0, 1.7320508075688772);
\draw[color=red] (4.5, 7.794228634059947)--(5.0, 8.660254037844386);
\draw[color=red] (10, 0)--(11, 0);
\draw[color=red] (6, 6.928203230275509)--(7, 6.928203230275509);
\draw[color=red] (9, 1.7320508075688772)--(10, 1.7320508075688772);
\draw[color=red] (5, 8.660254037844386)--(6, 8.660254037844386);
\draw[color=red] (11, 0)--(10.5, 0.8660254037844386);
\draw[color=red] (7, 6.928203230275509)--(6.5, 7.794228634059947);
\draw[color=red] (10, 1.7320508075688772)--(9.5, 2.598076211353316);
\draw[color=red] (6, 8.660254037844386)--(5.5, 9.526279441628825);
\draw[color=red] (10, 0)--(10.5, 0.8660254037844386);
\draw[color=red] (6, 6.928203230275509)--(6.5, 7.794228634059947);
\draw[color=red] (9, 1.7320508075688772)--(9.5, 2.598076211353316);
\draw[color=red] (5, 8.660254037844386)--(5.5, 9.526279441628825);
\draw[color=red] (11, 0)--(12, 0);
\draw[color=red] (7, 6.928203230275509)--(8, 6.928203230275509);
\draw[color=red] (10, 1.7320508075688772)--(11, 1.7320508075688772);
\draw[color=red] (6, 8.660254037844386)--(7, 8.660254037844386);
\draw[color=red] (10.5, 0.8660254037844386)--(11.5, 0.8660254037844386);
\draw[color=red] (6.5, 7.794228634059947)--(7.5, 7.794228634059947);
\draw[color=red] (9.5, 2.598076211353316)--(10.5, 2.598076211353316);
\draw[color=red] (5.5, 9.526279441628825)--(6.5, 9.526279441628825);
\draw[color=red] (12, 0)--(11.5, 0.8660254037844386);
\draw[color=red] (8, 6.928203230275509)--(7.5, 7.794228634059947);
\draw[color=red] (11, 1.7320508075688772)--(10.5, 2.598076211353316);
\draw[color=red] (7, 8.660254037844386)--(6.5, 9.526279441628825);
\draw[color=red] (11.5, 0.8660254037844386)--(11.0, 1.7320508075688772);
\draw[color=red] (7.5, 7.794228634059947)--(7.0, 8.660254037844386);
\draw[color=red] (10.5, 2.598076211353316)--(10.0, 3.4641016151377544);
\draw[color=red] (6.5, 9.526279441628825)--(6.0, 10.392304845413264);
\draw[color=red] (11, 0)--(11.5, 0.8660254037844386);
\draw[color=red] (7, 6.928203230275509)--(7.5, 7.794228634059947);
\draw[color=red] (10, 1.7320508075688772)--(10.5, 2.598076211353316);
\draw[color=red] (6, 8.660254037844386)--(6.5, 9.526279441628825);
\draw[color=red] (10.5, 0.8660254037844386)--(11.0, 1.7320508075688772);
\draw[color=red] (6.5, 7.794228634059947)--(7.0, 8.660254037844386);
\draw[color=red] (9.5, 2.598076211353316)--(10.0, 3.4641016151377544);
\draw[color=red] (5.5, 9.526279441628825)--(6.0, 10.392304845413264);
\draw[color=black] (12, 0)--(13, 0);
\draw[color=red] (8, 6.928203230275509)--(9, 6.928203230275509);
\draw[color=red] (10, 3.4641016151377544)--(11, 3.4641016151377544);
\draw[color=black] (6, 10.392304845413264)--(7, 10.392304845413264);
\draw[color=black] (13, 0)--(12.5, 0.8660254037844386);
\draw[color=red] (9, 6.928203230275509)--(8.5, 7.794228634059947);
\draw[color=red] (11, 3.4641016151377544)--(10.5, 4.330127018922193);
\draw[color=black] (7, 10.392304845413264)--(6.5, 11.258330249197702);
\draw[color=black] (12, 0)--(12.5, 0.8660254037844386);
\draw[color=red] (8, 6.928203230275509)--(8.5, 7.794228634059947);
\draw[color=red] (10, 3.4641016151377544)--(10.5, 4.330127018922193);
\draw[color=black] (6, 10.392304845413264)--(6.5, 11.258330249197702);
\draw[color=black] (13, 0)--(14, 0);
\draw[color=red] (9, 6.928203230275509)--(10, 6.928203230275509);
\draw[color=red] (11, 3.4641016151377544)--(12, 3.4641016151377544);
\draw[color=black] (7, 10.392304845413264)--(8, 10.392304845413264);
\draw[color=black] (12.5, 0.8660254037844386)--(13.5, 0.8660254037844386);
\draw[color=red] (8.5, 7.794228634059947)--(9.5, 7.794228634059947);
\draw[color=red] (10.5, 4.330127018922193)--(11.5, 4.330127018922193);
\draw[color=black] (6.5, 11.258330249197702)--(7.5, 11.258330249197702);
\draw[color=black] (14, 0)--(13.5, 0.8660254037844386);
\draw[color=red] (10, 6.928203230275509)--(9.5, 7.794228634059947);
\draw[color=red] (12, 3.4641016151377544)--(11.5, 4.330127018922193);
\draw[color=black] (8, 10.392304845413264)--(7.5, 11.258330249197702);
\draw[color=black] (13.5, 0.8660254037844386)--(13.0, 1.7320508075688772);
\draw[color=red] (9.5, 7.794228634059947)--(9.0, 8.660254037844386);
\draw[color=red] (11.5, 4.330127018922193)--(11.0, 5.196152422706632);
\draw[color=black] (7.5, 11.258330249197702)--(7.0, 12.12435565298214);
\draw[color=black] (13, 0)--(13.5, 0.8660254037844386);
\draw[color=red] (9, 6.928203230275509)--(9.5, 7.794228634059947);
\draw[color=red] (11, 3.4641016151377544)--(11.5, 4.330127018922193);
\draw[color=black] (7, 10.392304845413264)--(7.5, 11.258330249197702);
\draw[color=black] (12.5, 0.8660254037844386)--(13.0, 1.7320508075688772);
\draw[color=red] (8.5, 7.794228634059947)--(9.0, 8.660254037844386);
\draw[color=red] (10.5, 4.330127018922193)--(11.0, 5.196152422706632);
\draw[color=black] (6.5, 11.258330249197702)--(7.0, 12.12435565298214);
\draw[color=black] (14, 0)--(15, 0);
\draw[color=red] (10, 6.928203230275509)--(11, 6.928203230275509);
\draw[color=red] (12, 3.4641016151377544)--(13, 3.4641016151377544);
\draw[color=black] (8, 10.392304845413264)--(9, 10.392304845413264);
\draw[color=black] (13, 1.7320508075688772)--(14, 1.7320508075688772);
\draw[color=red] (9, 8.660254037844386)--(10, 8.660254037844386);
\draw[color=red] (11, 5.196152422706632)--(12, 5.196152422706632);
\draw[color=black] (7, 12.12435565298214)--(8, 12.12435565298214);
\draw[color=black] (15, 0)--(14.5, 0.8660254037844386);
\draw[color=red] (11, 6.928203230275509)--(10.5, 7.794228634059947);
\draw[color=red] (13, 3.4641016151377544)--(12.5, 4.330127018922193);
\draw[color=black] (9, 10.392304845413264)--(8.5, 11.258330249197702);
\draw[color=black] (14, 1.7320508075688772)--(13.5, 2.598076211353316);
\draw[color=red] (10, 8.660254037844386)--(9.5, 9.526279441628825);
\draw[color=red] (12, 5.196152422706632)--(11.5, 6.06217782649107);
\draw[color=black] (8, 12.12435565298214)--(7.5, 12.99038105676658);
\draw[color=black] (14, 0)--(14.5, 0.8660254037844386);
\draw[color=red] (10, 6.928203230275509)--(10.5, 7.794228634059947);
\draw[color=red] (12, 3.4641016151377544)--(12.5, 4.330127018922193);
\draw[color=black] (8, 10.392304845413264)--(8.5, 11.258330249197702);
\draw[color=black] (13, 1.7320508075688772)--(13.5, 2.598076211353316);
\draw[color=red] (9, 8.660254037844386)--(9.5, 9.526279441628825);
\draw[color=red] (11, 5.196152422706632)--(11.5, 6.06217782649107);
\draw[color=black] (7, 12.12435565298214)--(7.5, 12.99038105676658);
\draw[color=black] (15, 0)--(16, 0);
\draw[color=red] (11, 6.928203230275509)--(12, 6.928203230275509);
\draw[color=red] (13, 3.4641016151377544)--(14, 3.4641016151377544);
\draw[color=black] (9, 10.392304845413264)--(10, 10.392304845413264);
\draw[color=black] (14, 1.7320508075688772)--(15, 1.7320508075688772);
\draw[color=red] (10, 8.660254037844386)--(11, 8.660254037844386);
\draw[color=red] (12, 5.196152422706632)--(13, 5.196152422706632);
\draw[color=black] (8, 12.12435565298214)--(9, 12.12435565298214);
\draw[color=black] (14.5, 0.8660254037844386)--(15.5, 0.8660254037844386);
\draw[color=red] (10.5, 7.794228634059947)--(11.5, 7.794228634059947);
\draw[color=red] (12.5, 4.330127018922193)--(13.5, 4.330127018922193);
\draw[color=black] (8.5, 11.258330249197702)--(9.5, 11.258330249197702);
\draw[color=black] (13.5, 2.598076211353316)--(14.5, 2.598076211353316);
\draw[color=red] (9.5, 9.526279441628825)--(10.5, 9.526279441628825);
\draw[color=red] (11.5, 6.06217782649107)--(12.5, 6.06217782649107);
\draw[color=black] (7.5, 12.99038105676658)--(8.5, 12.99038105676658);
\draw[color=black] (16, 0)--(15.5, 0.8660254037844386);
\draw[color=red] (12, 6.928203230275509)--(11.5, 7.794228634059947);
\draw[color=red] (14, 3.4641016151377544)--(13.5, 4.330127018922193);
\draw[color=black] (10, 10.392304845413264)--(9.5, 11.258330249197702);
\draw[color=black] (15, 1.7320508075688772)--(14.5, 2.598076211353316);
\draw[color=red] (11, 8.660254037844386)--(10.5, 9.526279441628825);
\draw[color=red] (13, 5.196152422706632)--(12.5, 6.06217782649107);
\draw[color=black] (9, 12.12435565298214)--(8.5, 12.99038105676658);
\draw[color=black] (15.5, 0.8660254037844386)--(15.0, 1.7320508075688772);
\draw[color=red] (11.5, 7.794228634059947)--(11.0, 8.660254037844386);
\draw[color=red] (13.5, 4.330127018922193)--(13.0, 5.196152422706632);
\draw[color=black] (9.5, 11.258330249197702)--(9.0, 12.12435565298214);
\draw[color=black] (14.5, 2.598076211353316)--(14.0, 3.4641016151377544);
\draw[color=red] (10.5, 9.526279441628825)--(10.0, 10.392304845413264);
\draw[color=red] (12.5, 6.06217782649107)--(12.0, 6.928203230275509);
\draw[color=black] (8.5, 12.99038105676658)--(8.0, 13.856406460551018);
\draw[color=black] (15, 0)--(15.5, 0.8660254037844386);
\draw[color=red] (11, 6.928203230275509)--(11.5, 7.794228634059947);
\draw[color=red] (13, 3.4641016151377544)--(13.5, 4.330127018922193);
\draw[color=black] (9, 10.392304845413264)--(9.5, 11.258330249197702);
\draw[color=black] (14, 1.7320508075688772)--(14.5, 2.598076211353316);
\draw[color=red] (10, 8.660254037844386)--(10.5, 9.526279441628825);
\draw[color=red] (12, 5.196152422706632)--(12.5, 6.06217782649107);
\draw[color=black] (8, 12.12435565298214)--(8.5, 12.99038105676658);
\draw[color=black] (14.5, 0.8660254037844386)--(15.0, 1.7320508075688772);
\draw[color=red] (10.5, 7.794228634059947)--(11.0, 8.660254037844386);
\draw[color=red] (12.5, 4.330127018922193)--(13.0, 5.196152422706632);
\draw[color=black] (8.5, 11.258330249197702)--(9.0, 12.12435565298214);
\draw[color=black] (13.5, 2.598076211353316)--(14.0, 3.4641016151377544);
\draw[color=red] (9.5, 9.526279441628825)--(10.0, 10.392304845413264);
\draw[color=red] (11.5, 6.06217782649107)--(12.0, 6.928203230275509);
\draw[color=black] (7.5, 12.99038105676658)--(8.0, 13.856406460551018);

\node[shape=circle,draw=black,fill=white] at (0,0){$o$};
\node[shape=circle,draw=black,fill=white] at (4,0){$x_2$};
\node[shape=circle,draw=black,fill=white] at (2,3.464101){$y_2$};
\node[shape=circle,draw=black,fill=white] at (16,0){$x_4$};
\node[shape=circle,draw=black,fill=white] at (14,3.464101){$a_2$};
\node[shape=circle,draw=black,fill=white] at (12,0){$b_2$};
\node[shape=circle,draw=black,fill=white] at (8,13.8564){$y_4$};
\node[shape=circle,draw=black,fill=white] at (6,10.3923){$c_2$};
\node[shape=circle,draw=black,fill=white] at (10,10.3923){$d_2$};
\end{tikzpicture}
\caption{An illustration of $\mathsf{SG}_{2,4}$. The picture shows the graph $\mathsf{SG}_4^+$ with its three corner vertices $o,x_4$ and $y_4$. The part colored in red shows $\mathsf{SG}_{2,4}$ with its six corner vertices $\{x_2,y_2,a_2,b_2,c_2,d_2\}$.}
\label{fig:inner_hex}
\end{figure}

\begin{lemma}\label{lem1}
For any sufficiently large $n$ and any $m<n-1$ we have
$$\mathbb{P}\Big(N_{m,n}\geq 3|V_{m,n}|+\sigma_0\sqrt{|V_{m,n}|}\Big)\geq 0.15.$$
\end{lemma}
\begin{proof}
Let $\delta\geq 0$ be such that $\mathbb{E}[\sigma(o)]=3+\delta$. Suppose first that $\delta>0$, i.e.~$\mathbb{E}[\sigma(0)]>3$. For $n$ sufficiently large we have that $\delta |V_{m,n}|\geq \sigma_0\sqrt{V_{m,n}}$. Moreover, for every $m,n\in\N$ with $m<n-1$ the total number $N_{m,n}$ of particles in $\SG_{m,n}$ is given by a sum of $|V_{m,n}|$ i.i.d.~random variables with finite variance $\sigma_0^2>0$, so  by the central limit theorem, the random variable 
$\frac{N_{m,n}-(3+\delta)|V_{m,n}|}{\sigma_0\sqrt{|V_{m,n}|}}$ converges in distribution to a standard normal distribution $\mathcal{N}(0,1)$ as $n\to\infty$. Thus, for $n$ sufficiently large, we have 
$$\mathbb{P}\Big(N_{m,n}\geq 3|V_{m,n}|+\sigma_0\sqrt{|V_{m,n}|}\Big)\geq \Big(N_{m,n}\geq (3+\delta)|V_{m,n}|\Big)\geq 0.49.$$ 
If $\delta=0$, then for sufficiently large $n$, we have
$$\mathbb{P}\Big(N_{m,n}\geq 3|V_{m,n}|+\sigma_0\sqrt{|V_{m,n}|}\Big)\geq 1-\Phi(1)-0.0001\geq 0.15,$$
where the first inequality is again due to the central limit theorem, and this completes the claim of the proof.
\end{proof}

\begin{lemma}\label{lem2}
If we denote by $u_{m,n}$ the odometer function for the stabilisation of the i.i.d.~sandpile $\sigma$ on $\SG_{m,n}$, then for all sufficiently large $n$ and $m<n-1$ we have
$$\mathbb{P}\Big(u_{m,n}(x_m)\geq \frac{\sigma_0}{6}\sqrt{|V_{m,n}|}\Big)\geq 0.025.$$
\end{lemma}
\begin{proof}
On the event that $\{N_{m,n}\geq 3|V_{m,n}|+\sigma_0\sqrt{|V_{m,n}|}\}$, which in view of Lemma \ref{lem1} occurs with probability at least $0.15$, for $n$ large enough, we know that $\sigma$ restricted to $\SG_{m,n}$ is unstable, so the excess mass of at least $\sigma_0\sqrt{|V_{m,n}|}$ has to leave $\SG_{m,n}$ through the six corner vertices $\{x_m,y_m,a_m,b_m,c_m,d_m\}$. Because of the symmetry of $\SG_{m,n}$, in at least one-sixth of the time, $x_m$ has the largest amount of mass among these six corners, which together with  Lemma \ref{lem1} proves the claim.
\end{proof}

For $m,n\in\mathbb{N}$ with $m<n$, denote by $B_{m,n}$ the following event:
$$B_{m,n}:=\Big\{u_{m,n}(x_m)\geq \frac{\sigma_0}{6}\sqrt{|V_{m,n}|} \Big\}.$$

\begin{lemma}\label{lem3}
Let $n_1,n_2,\ldots$ be any strictly increasing sequence of positive integers. Then
$$\mathbb{P}\big(B_{n_i,n_{i+1}} \text{occurs for infinitely many }i\big)=1.$$
\end{lemma}
\begin{proof}
From the construction, for $i\neq j$, the graphs $\SG_{n_i,n_{i+1}}$ and $\SG_{n_j,n_{j+1}}$ have disjoint vertices and the sandpile configuration $\sigma$ restricted on $\SG_{n_i,n_{i+1}}$ in independent of $\sigma$ restricted on $\SG_{n_j,n_{j+1}}$. Thus the events $(B_{n_i,n_{i+1}})_{i\in\mathbb{N}}$ are independent with $\mathbb{P}(B_{n_i,n_{i+1}})\geq 0.025$ for every $i\in\mathbb{N}$ in view of Lemma \ref{lem2}, which together with Borel-Cantelli proves the claim. 
\end{proof}
The previous three results lead immeadiately to the proof of Theorem \ref{thm:iid-sandpile}.

\begin{proof}[Proof of Theorem \ref{thm:iid-sandpile}]
The first step of the proof is to apply Lemma \ref{lem3} to the following particular sequence $(n_i)_{i\in\mathbb{N}}$ of positive integers defined recursively as follows. Let $n_1=1$, and for any $i\geq 2$ let $n_{i+1}$ be defined as follows:
$$n_{i+1}:=\min\Big\{k>n_i+1:\ \frac{\sigma_0}{6}\sqrt{|V_{n_i,k}|}\geq i\cdot 3^{n_i} \Big\}.$$
From the definition and the cutpoint structure of the gasket, it is clear that such a sequence is well defined. Indeed, since $|V_n|=\frac32(3+1)$ and  for any $m<n-1$ we have  $|V_{m,n}|=|V_n|-3|V_m|+6$, one can compute each term of the sequence $(n_i)_{i\in\mathbb{N}}$ explicitely.

Now suppose that the event $B_{n_i,n_{i+1}}$ occurs for a specific $i\geq 1$. Then, from the definition of the event $B_{n_i,n_{i+1}}$ and the recursive choice of $n_i$, we have
$$u_{n_i,n_{i+1}}(x_{n_i})\geq \frac{\sigma_0}{6}\sqrt{|V_{n_i,n_{i+1}}|}\geq i\cdot 3^{n_i},$$
i.e.~the vertex $x_{n_i}$ topples at least $i\cdot 3^{n_i}$ times during the stabilisation of $\sigma$ on $\SG_{n_i,n_{i+1}}$. By the Abelian property, this imples that for all $k\leq n_i$, the vertex $x_k$ has been toppled at least $i\cdot 3^k$ times.  In particular, this implies that $x_0=o$ has been toppled at least $i$ times if $B_{n_i,n_{i+1}}$ occurs. On the other hand, by Lemma \ref{lem3} we know that $B_{n_i,n_{i+1}}$ occurs infinitely often with probability $1$, thus $o$ topples infinitely often with probability $1$, i.e.~ $\mathbb{P}(u^{\infty}(0)=\infty)=1$ and this completes the proof.
\end{proof}

\begin{remark}
For an i.i.d.~sandpile $\sigma$ with $ \E[\sigma(o)]>3$ and $\mathsf{Var}[\sigma(o)]=\infty$, the fact that $\sigma$ cannot be stabilized follows immediately. If $\sigma$ is not bounded, then we can consider the bounded sandpile $\sigma\mathds{1}_{\{\sigma\leq N\}}=\left(\sigma(x)\mathds{1}_{\{\sigma(x)\leq N\}}\right)_{v\in\SG}$ for some $N\in\N$ large enough and use that $\sigma(x)\geq \sigma(x)\mathds{1}_{\{\sigma(x)\leq N\}}$ for all $x\in\SG$ almost surely. Then by Theorem \ref{thm:iid-sandpile}, $\sigma\mathds{1}_{\{\sigma\leq N\}}$ is not stabilizable almost surely, and together with Remark \ref{rem:monotonicity_stab} we get that $\sigma$ is also not stabilizable in the unbounded case.
\end{remark}
\vspace{-0.3em}
The same proof technique as in Theorem \ref{thm:iid-sandpile} can be used to prove explosion of i.i.d.~sandpiles $\sigma$ on any infinite graph $G$ for which: 
\vspace{-0.4em}
\begin{enumerate}[(a)]
\setlength\itemsep{-0.1em}
\item There exists an exhaustion $G_1\subset G_2 \subset ...$ of $G$, and for all $n\in \N$, $G_n$ is finite.
\item There exists a vertex $v\in G$, and an increasing sequence $(H_n)_{n}$ of subgraphs with $H_n\subseteq G_n$ such that the boundary $\partial H_n$ of $H_n$ is symmetric and $v\in \partial H_n$. 
\item There exists a constant $C>0$ independent of $n$, such that we have for all $n\in\N$ that $|\partial H_n|<C$.
\end{enumerate}
This properties hold for finitely ramified fractal graphs. They do not apply to $\Z^2$ when stabilizing for instance in nested boxes $[-n,n]^2$, since  the boundary of the box $[-n,n]^2$ is not symmetric and it grows linearly in $n$ while for the size of the box we have quadratic growth in $n$: $|[-n,n]^2|=(2n+1)^2$. 

\subsubsection*{Divisible sandpiles and final remarks}

We recall here another model of mass redistribution which is very similar to the Abelian sandpile model and is called {\it the divisible sandpile model}. In the divisible sandpile as introduced in \cite{divisible-sandpile-levine-peres}, we have a divisible sand configuration $\sigma:V\to\mathbb{R}$ on a graph $G=(V,E)$, which indicates the amount of mass present at each vertex. Contrary to the Abelian sandpile model, one allows fractional mass to be distributed and a vertex $x\in G$ is unstable if $\sigma(x)>1$; in this case the excess mass $\sigma(x)-1$ is distributed equally among the neighbours. There are several toppling procedures like in the case of the Abelian sandpile model; see \cite[Proposition 2.5]{div-sand-crit-dens} for details on legal toppling procedures. The initial divisible sand configuration $\sigma$ may be random or not; if $(\sigma(x))_{x\in V}$ is a sequence of i.i.d.~random variables, then we call the model {\it i.i.d.~divisible sandpile model}. Then we can define stabilization (in infinite volume) of $\sigma$ exactly the same way as in the case of Abelian sandpiles: if $u_n$ is the amount of mass emitted from a vertex during the stabilization of $\sigma$ on $G_n$, then $\sigma$ stabilizes if $u^{\infty}(x)<\infty$ for all $x$, where $u^{\infty}=\lim_{n\to\infty}u_n$. The restriction in the Abelian sandpile that the odometer function has to be integer valued introduces difficulties that are not present in the divisible sandpile model. 
In \cite{div-sand-crit-dens}, the authors investigate i.i.d.~divisible sandpiles at critical density (i.e.~when the expected number of chips per site equals one) on $\Z^d$, and more generally on infinite vertex-transitive graphs, and they show that if the initial masses have finite variances, then the i.i.d.~divisible sandpile does not stabilize almost surely. They also ask for which infinite graphs there exist i.i.d.~divisible sandpiles at critical density that do stabilize almost surely. We contribute in this direction by discarding the case of $\SG$ on which an i.i.d.~divisible sandpile at critical density does not stabilize almost surely, since the proof of Theorem \ref{thm:iid-sandpile} carries over to i.i.d.~divisible sandpiles on the  Sierpi\'nski gasket graph $\SG$.
Therefore, we have the following.

\begin{proposition}\label{prop:div}
Let $\sigma$ be an i.i.d.~divisible sandpile on the infinite Sierpi\'nski gasket graph $\SG$ with $\mathbb{E}[\sigma(0)]\geq 1 $ and $0<\mathsf{Var}[\sigma(0)]<\infty$. Then $\sigma$ does not stabilize almost surely.
\end{proposition}
We omit the proof since it is identical with the one of Theorem \ref{thm:iid-sandpile}, the minor difference being that one has to replace the expectation $\E[\sigma(0)]\geq 3$ with $\E[\sigma(0)]\geq 1$, therefore the random variable
$\frac{N_n-(1+\delta)|V_n|}{\sigma_0\sqrt{|V_n|}}$ converges in distribution to $\mathcal{N}(0,1)$ as $n\to\infty$ in the proof of Lemma \ref{lem1}.

\textbf{Questions on rotor walks.}
While in this work we have proven the recurrence of rotor walks with random initial configuration of rotors,
it is interesting to understand "how recurrent" these walks are, when compared to simple random walks on $\SG$. If we denote by $\mathsf{G}^{\mathsf{SRW}}(x,y)$ (respectively $\mathsf{G}^{\mathsf{URW}}(x,y)$) the expected number of visits to $y$ of a simple random walk (respectively of a uniform rotor walk) starting at $x$, what can we say about the limit
$\lim_{n\to\infty}\frac{\mathsf{G}^{\mathsf{SRW}}(x,y)}{\mathsf{G}^{\mathsf{URW}}(x,y)}?$ Does it exist?
For graphs $G$ on which the simple random walk is transient, results comparing the Green's functions $\mathsf{G}^{\mathsf{SRW}}(x,y)$ and $\mathsf{G}^{\mathsf{URW}}(x,y)$ for the simple random walk and the random rotor walk are given in \cite{swee-hong1}; see also \cite{swee-hong2}.

\textbf{Questions on Abelian and divisible sandpiles.}
The proof of stabilization of sandpiles relies mostly on the fact that for all $A>0$, $\liminf_n \mathbb{P}(|N_n-\E[N_n]|>A)>0$. If we consider the distribution of $\sigma(x)$, for $x\in\SG$, to lie in the domain of attraction of a stable law with index strictly greater than one, one can adapt the proof of Theorem \ref{thm:iid-sandpile} and of the Lemma \ref{lem1} with a stable law in place of the normal distribution, and with weaker conditions than the finiteness of the variance, but the argument might become pretty technical.
What about i.i.d.~sandpiles $\sigma$ on $\SG$ with $\E[\sigma(o)]=3$ and $\mathsf{Var}[\sigma(o)]=\infty$? Do they explode almost surely? Similarly, what can one say about i.i.d.~divisible sandpiles on $\SG$ with  $\E[\sigma(o)]=1$ and $\mathsf{Var}[\sigma(o)]=\infty$? 

\vspace{-0.2cm}

\subsection*{Acknowledgments} 
The research of both authors is supported by the Austrian Science Fund (FWF): P 34129. We are very grateful to the referee for finding a gap in the first version on the paper, and for suggesting how to fix the gap in order to prove the claim of Theorem \ref{thm:iid-sandpile}. In particular, the construction of the subgraphs $\SG_{m,n}$, of the sequence $n_i$, and the use of the Borel-Cantelli Lemma for the events $B_{n_i,n_{i+1}}$ has been suggested by the referee, whose several other suggestions improved the quality of the paper.
\bibliography{references}{}
\bibliographystyle{alpha}

\textsc{Robin Kaiser}, Institut für Mathematik, Universität Innsbruck, Austria.\\
\texttt{Robin.Kaiser@uibk.ac.at}

\textsc{Ecaterina Sava-Huss}, Institut für Mathematik, Universität Innsbruck, Austria.\\
\texttt{Ecaterina.Sava-Huss@uibk.ac.at}
\end{document}